\pdfoutput=1

\documentclass[11 pt, letterpaper, oneside]{article}
\usepackage[T1]{fontenc}
\usepackage{inputenc,xcolor,amsthm,hyperref, cite, color, url,amsfonts,amsmath,amssymb,caption,multicol,graphicx,fancyhdr,tikz,pgfplots,subcaption,comment,commath,bm}
\usepackage{pdflscape}
\usepgfplotslibrary{groupplots}
\usepackage[toc,page]{appendix}
\usepackage[bottom=3cm, headsep=1cm,headheight=2cm]{geometry}

\usepackage{tikz}
\usetikzlibrary {positioning}
\usetikzlibrary{automata}
\definecolor {processblue}{cmyk}{1,1,1,1}
\tikzset{state/.style ={circle,draw,black,text=black,minimum width =1 cm}}
\tikzset{box/.style   ={rectangle, draw,black,text=black,minimum width =1 cm}}
\tikzstyle{phase} = [fill,shape=circle,minimum size=5pt,inner sep=0pt]

  \usepackage{pgfplots}
  \pgfplotsset{compat=newest}
  \usetikzlibrary{plotmarks}
  \usepackage{grffile}
  \usepackage{amsmath}
\DeclareMathOperator{\E}{\mathbb{E}}
\DeclareMathOperator{\Pb}{\mathbb{P}}

\DeclareMathOperator{\R}{\mathbb{R}}
\DeclareMathOperator{\Var}{\mathbb{V}ar}
\DeclareMathOperator{\Corr}{\mathbb{C}orr}
\DeclareMathOperator{\Cov}{\mathbb{C}ov}
\DeclareMathOperator{\F}{\mathcal{F}}

\newcommand{\D}{\Delta}

\renewcommand{\b}{\beta}
\renewcommand{\d}{\delta}
\renewcommand{\l}{\lambda}
\newcommand{\m}{\mu}

\newcommand{\G}{\Gamma}
\newcommand{\s}{\sigma}

\newcommand{\sq}{\sqrt}
\newcommand{\f}[2]{\frac{#1}{#2}}

\renewcommand{\L}{\Lambda}
\newcommand{\p}{\partial}

\usepackage{blindtext}
\usepackage{soul}

\newtheorem{theorem}{Theorem}[section]
\newtheorem{lemma}[theorem]{Lemma}
\newtheorem{proposition}[theorem]{Proposition}

\theoremstyle{definition}
\newtheorem{definition}[theorem]{Definition}

\newtheorem{remark}[theorem]{Remark}
\newtheorem*{remark*}{Remark}
\newtheorem{example}[theorem]{Example}

\newcommand{\pushright}[1]{\ifmeasuring@#1\else\omit\hfill$\displaystyle#1$\fi\ignorespaces}

\newcommand{\Int}{\int\limits}

\usepackage[affil-it]{authblk}
\title{Networks of $\cdot/G/\infty$ Queues\\ with Shot-Noise-Driven Arrival Intensities}
\author[1]{D.T. Koops}
\author[2]{O.J. Boxma}
\author[1]{M.R.H. Mandjes}
\affil[1]{Korteweg-de Vries Institute\\ University of Amsterdam}
\affil[2]{Eurandom and Department of Mathematics and Computer Science\\ Eindhoven University of Technology}
\date{\today}

\newcommand{\vb}{\vspace{3mm}}

\usepackage{pgfplots}
\usepackage{mathtools}
\usepackage{tikz}

\begin{document}
\maketitle

\begin{abstract}
\noindent We study infinite-server queues in which the arrival process is a Cox process (or doubly stochastic Poisson process), of which the arrival rate is given by a shot-noise process. A shot-noise rate emerges naturally in cases where the arrival rate tends to exhibit sudden increases (or: shots) at random epochs, after which the rate is inclined to revert to lower values. Exponential decay of the shot noise is assumed, so that the queueing systems are amenable to analysis. In particular, we perform transient analysis on the number of jobs in the queue jointly with the value of the driving shot-noise process. Additionally, we derive heavy-traffic asymptotics for the number of jobs in the system by using a linear scaling of the shot intensity. First we focus on a one dimensional setting in which there is a single infinite-server queue, which we then extend to a network setting. 
\end{abstract}

\section{Introduction}
\label{sec:introduction}
In the queueing literature one has traditionally studied queues with Poisson input. The Poisson assumption typically facilitates explicit analysis, but it does not always align well with actual data, see e.g.\ \cite{KW2014} and references therein. More specifically, statistical studies show that in many practical situations, Poisson processes underestimate the variability of the queue's input stream.
%It is unreasonable to expect that queueing models fed by homogeneous Poisson input perform properly, while the assumption on the arrival process is violated. 
This observation has motivated research on queues fed by arrival processes that better capture the burstiness observed in practice.

\vb

The extent to which burstiness takes place can be measured by the dispersion index, i.e.\ the ratio of the variance of the number of arrivals in a given interval, and the corresponding expected value. In arrival streams that display burstiness, the dispersion index is larger than unity (as opposed to Poisson processes, for which it is equal to unity), a phenomenon that is usually referred to as {\it overdispersion}. 
\begin{comment}
In the literature there has been a myriad of efforts to realistically model arrival streams. 
In various papers the use of {\it time-inhomogeneous Poisson processes} has been advocated. In such processes the Poisson arrival rate is time-dependent but still deterministic, say with value  $\lambda(t)$ at time $t$. The function $\lambda(\cdot)$ could for instance be chosen such that it captures a certain periodicity (to model e.g.\ daily or weekly patterns). Models of this type have been studied in a queueing context in e.g.\ \cite{EMW1993}. It is noted, however, that the for these time-inhomogeneous Poisson processes the number of arrivals in a given time interval still has a Poisson distribution, and hence this class of models fails to incorporate overdispersion. 
\end{comment}
It is desirable that the arrival process of the queueing model takes the observed overdispersion into account. One way to achieve this, is to make use of {\it Cox processes}, which are Poisson processes, conditional on the stochastic time-dependent intensity. It is an immediate consequence of the law of total variance, that Cox processes \textit{do} have a dispersion index larger than unity. Therefore, this class of processes makes for a good candidate to model overdispersed input processes.

\vb

In this paper we contribute to the development of queueing models fed by input streams that exhibit overdispersion. We analyze infinite-server queues driven by a particular Cox process, in which the rate is a (stochastic) shot-noise process. 

The shot-noise process we use is one in which there are only upward jumps (or: shots), that arrive according to a homogeneous Poisson process. Furthermore, we employ an exponential `response' or `decay' function, which encodes how quickly the process will decline after a jump. In this case, the shot-noise process is a Markov process, see \cite[p.\ 393]{Ross}. There are several variations on shot-noise processes; see e.g. \cite{IJ2003} for a comprehensive overview.

\vb

It is not a novel idea to use a shot-noise process as stochastic intensity. For instance, in insurance mathematics, the authors of \cite{D2003} use a shot-noise-driven Cox process to model the claim count. They assume that disasters happen according to a Poisson process, and each disaster can induce a cluster of arriving claims. The disaster corresponds to a shot upwards in the claim intensity. As time passes, the claim intensity process decreases, as more and more claims are settled.  

Another example of shot-noise arrival processes is found in the famous paper \cite{O1989}, where it is used to model the occurences of earthquakes.  The arrival process considered in \cite{O1989} has one crucial difference with the one used in this paper: it makes use of Hawkes processes \cite{Hawkes}, which do have a shot-noise structure, but have the special feature that they are {\it self-exciting}. More specifically, in Hawkes processes, an arrival induces a shot in the arrival rate, whereas in our shot-noise-driven Cox model these shots are merely exogenous. The Hawkes process is less tractable than the shot-noise-driven Cox process.  A very recent effort to analyze $\cdot/G/\infty$ queues that are driven by a Hawkes process has been made in \cite{Gao2016}, where a functional central limit theorem is derived for the number of jobs in the system. In this model, obtaining explicit results (in a non-asymptotic setting), as we are able to do in the shot-noise-driven Cox variant, is still an open problem. 
%To the best of the authors' knowledge, this is the only application of a shot-noise type arrival process to queueing theory, despite its widespread appearance in many other scientific fields.

\vb
\begin{comment}
Even though shot-noise processes are not new in queueing theory, they have never been regarded as a stochastic arrival rate, to the best of our knowledge. A shot-noise process appears in a queueing context, if e.g.\ one considers an $M/G/1$ server which has a linear service rate depending on the current workload. For literature on this, consider \cite{KellaWhitt1999} and its follow up papers. Furthermore, storage models with perishable goods also admit shot-noise forms, see \cite{Britt}. 
\end{comment}

In order to successfully implement a theoretical model, it is crucial to have methods to estimate its parameters from data. The shot-noise-driven Cox process is attractive since it has this property. Statistical methods that filter the unobservable intensity process, based on Markov Chain Monte Carlo (MCMC) techniques, have been developed; see \cite{Centanni2006} and references therein. By filtering, they refer to the estimation of the intensity process in a given time interval, given a realized arrival process. Subsequently, given this `filtered path' of the intensity process, the parameters of the shot-noise process can be estimated by a Monte Carlo version of the expectation maximization (EM) method.
Furthermore, the shot-noise-driven Cox process can also be easily simulated;  see e.g.\ the thinning procedure described in \cite{sim}.

\vb

In this paper we study networks of infinite-server queues with shot-noise-driven Cox input. We assume that the service times at a given node are i.i.d.\ samples from a general distribution. The output of a queue is routed to a next queue, or leaves the network. Infinite-server queues have the inherent advantage that jobs do not interfere with one another, which considerably simplifies the analysis. Furthermore, infinite-server systems are frequently used to produce approximations for corresponding finite-server systems. In the network setting, we can model queueing systems that are driven by correlated shot-noise arrival processes. With regards to applications, such a system could, e.g., represent the call centers of a fire department and police department in the same town.

\vb

The contributions and organization of this paper are as follows.
In this paper we derive exact and asymptotic results. The main result of the exact analysis is Thm.\ \ref{thm:main}, where we find the joint Laplace transform of the numbers of jobs in the queues of a feedforward network, jointly with the  shot-noise-driven arrival rates. We build up towards this result as follows. In Section \ref{sec:Notation and preliminaries} we introduce notation  and we state the important Lemma \ref{lemma} that we repeatedly rely on. Then we derive exact results for the single infinite-server queue with a shot-noise arrival rate, in Section \ref{sec: Exact analysis}. Subsequently, in Section \ref{sec:asymptotic analysis}, we show that after an appropriate scaling the number of jobs in the system satisfies a functional central limit theorem (Thm.\ \ref{thm:FCLT}); the limiting process is an Ornstein-Uhlenbeck (OU) process driven by a superposition of a Brownian motion and an integrated OU process. We then  extend the theory to a network setting in Section \ref{sec:Networks}. Before we consider full-blown networks, we first consider a tandem system consisting of an arbitrary number of infinite-server queues in Section \ref{sec:tandem}. Then it is argued in Section \ref{sec:parallel} that a feedforward network can be seen as a number of  tandem queues in parallel. We analyze two different ways in which dependency can enter the system through the arrival process. Firstly, in Model (M1), parallel service facilities are driven by a multidimensional shot-noise process in which the \textit{shots} are simultaneous (which includes the possibility that all shot-noise processes are equal). Secondly, in Model (M2), we assume that there is one shot-noise arrival intensity that generates simultaneous \textit{arrivals} in all tandems. In Section \ref{sec:concluding remarks} we finish with some concluding remarks.

\section{Notation and preliminaries}
\label{sec:Notation and preliminaries}
Let $(\Omega, \mathcal{F}, \{\mathcal{F}_t\}_{t\geq0},\Pb)$ be a probability space, in which the filtration $\{\mathcal{F}_t\}_{t\geq0}$ is such that $\L(\cdot)$ is adapted to it. A shot-noise process is a process that has random jumps at Poisson epochs, and a deterministic `response' or `decay' function, which governs the behavior of the process. See \cite[Section 8.7]{Ross} for a brief account of shot-noise processes. The shot noise that we use in this paper has the following representation:
\begin{equation}
\label{eq:SN}
\L(t) = \L(0)e^{-rt} + \sum_{i=1}^{P_B(t)} B_i e^{-r(t-t_i)},
\end{equation}
where the $B_i\geq0$ are i.i.d.\ shots from a general distribution, the decay function is exponential with rate $r>0$, $P_B$ is a homogeneous Poisson process with rate $\nu$, and the epochs of the shots, that arrived before time $t$, are labelled $t_1,t_2,\ldots,t_{P_B(t)}$. 

As explained in the introduction, the shot-noise process serves as a stochastic arrival rate to a queueing system. It is straightforward to simulate a shot-noise process; for an illustration of a sample path, consider Fig.\ \ref{fig:sn}. Using the thinning method for nonhomogeneous Poisson processes  \cite{sim}, and using the sample path of Fig. \ref{fig:sn} as the arrival rate, one can generate a corresponding sample path for the arrival process, as is displayed in Fig.\ \ref{fig:arrivals}. Typically, most arrivals occur shortly after peaks in the shot-noise process in Fig.\ \ref{fig:sn}, as expected.

\begin{figure}[ht!]
\centering
\begin{tikzpicture}
 \begin{axis}[
   ticks=none,
   scaled ticks=false,
   xmin=1,
   xmax=13,
   ymin=0,
   xlabel=$t$,
   ylabel=$\L(t)$,
   width=14cm,
   height=6cm,
   axis y line=left,
   axis x line=bottom
   ]
    \addplot[domain=0:1, black, thick, smooth] {0};
    \addplot[black, thick, dashed, smooth]coordinates{(1,0)(1,3)};
    \addplot[domain=1:3, black, thick,smooth] {3*exp(-(x-1))};
    \addplot[black, thick, dashed, smooth]coordinates{(3,3*0.13533528323)(3,3*0.13533528323+2)};
    \addplot[domain=3:4, black,thick,smooth]{2.40600584969*exp(-(x-3))};
    \addplot[black, thick, dashed, smooth]coordinates{(4,0.88512008743)(4,1.18512008743)};
    \addplot[domain=4:8, black,thick,smooth]{1.18512008743*exp(-(x-4))};
    \addplot[black, thick, dashed, smooth]coordinates{(8,0.02170623156)(8,0.02170623156+1.5)};
    \addplot[domain=8:9.75, black,thick,smooth]{(0.02170623156+1.5)*exp(-(x-8))};
    \addplot[black, thick, dashed, smooth]coordinates{(9.75,0.26443289263)(9.75,0.26443289263+3.5)};
    \addplot[domain=9.75:16, black,thick,smooth]{(0.26443289263+3.5)*exp(-(x-9.75))};
\end{axis}
\end{tikzpicture}
\caption{\textit{Sample path of shot-noise process}}
\label{fig:sn}

\centering
\begin{tikzpicture}

\begin{axis}[
   ticks=none,
   scaled ticks=false,
   xmin=1.35,
   xmax=13,
   ymin=0,
   xlabel=$t$,
   ylabel=number of arrivals,
   width=14cm,
   height=6cm,
   axis y line=left,
   axis x line=bottom
   ]   
   \addplot[thick,black]
    table {%
0 0
1.81 0
1.82 1
1.83 2
1.98 2
1.99 3
2.26 3
2.27 4
2.3 4
2.31 5
2.48 5
2.49 6
3.33 6
3.34 7
3.41 7
3.42 8
3.44 8
3.45 9
3.73 9
3.74 10
4.17 10
4.18 11
4.49 11
4.5 12
4.76 12
4.77 13
5.04 13
5.05 14
6.43 14
6.44 15
8.17 15
8.18 16
8.33 16
8.34 17
8.96 17
8.97 18
10.13 18
10.14 19
10.28 19
10.29 20
10.3 21
10.37 21
10.38 22
10.45 22
10.46 23
10.58 23
10.59 24
10.75 24
10.76 25
10.86 25
10.87 26
10.88 27
10.92 27
10.93 28
11.02 28
11.03 29
11.04 29
11.05 30
11.23 30
11.24 31
11.71 31
11.72 32
12.19 32
12.2 33
12.21 34
12.44 34
12.45 35
12.99 35
};
\end{axis}
\end{tikzpicture}%
\caption{\textit{A realization of arrival process corresponding to the sample path of the arrival rate process in Fig.\ \ref{fig:sn}}}
\label{fig:arrivals}
\end{figure}

\vb

We write $\L$ (i.e., without argument) for a random variable with distribution equal to that of $\lim_{t\to\infty}\L(t)$. We now present  well-known transient and stationary moments of the shot-noise process, see Appendix \ref{app:cov} and e.g.\ \cite{Ross}: with $B$ distributed as $B_1$,
\begin{align}
\label{eq:momentsSN}
\nonumber\E \L(t) = \L(0)e^{-rt} + \f{\nu\E B}{r}(1-e^{-rt})&, \quad\E\L=\f{\nu\E B}{r},\\
\Var \L(t)= \f{\nu\E B^2}{2r}(1-e^{-2rt})&, \quad\Var \L = \f{\nu\E B^2}{2r}, \\
\nonumber\Cov(\L(t),\L(t+\d))= e^{-r\d} \Var \L(t)&.
\end{align}

We remark that, for convenience, we throughout assume $\Lambda(0)=0$. The results can be readily extended to the case in which $\L(0)$ is a non-negative random variable, at the cost of a somewhat more cumbersome notation.

\vb

In the one-dimensional case, we denote 
$
\beta(s) = \E e^{-s B} 
$
, and in the multidimensional case, where $s=(s_1,s_2,\ldots,s_d)$, for some integer $d\geq2$, now denotes a vector, we write
\[
\beta(s) = \E e^{-\sum_i s_i B_i}.
\]
The following lemma will be important for the derivation of the joint transform of $\Lambda(t)$ and the number of jobs in system, in both single and multi-node cases. 
\begin{lemma}
\label{lemma}
Let $\L(\cdot)$ be a shot-noise process. Let $f:\R\times\R^{d}\to\R$ be a function which is piecewise continuous in its first argument, with at most a countable number of discontinuities. Then it holds that
\begin{align*}
&\E \exp\left(\int_0^t f(u,z) \L(u) \dif u - s \L(t)\right)\\
&= \exp\left(\nu \int_0^t\left( \beta\left(s e^{-r(t-v)}-e^{rv}\int_{v}^t f(u,z) e^{-ru}\dif u \right) -1\right)\dif v\right).
\end{align*}
\end{lemma}
\begin{proof}
See appendix \ref{app:proof}.
\end{proof}

\section{A single infinite-server queue}
In this section we study the $M_{\rm S}/G/\infty$ queue. This is a single infinite-server queue, of which the arrival process is a Cox process driven by the shot-noise process $\L(\cdot)$, as defined in Section \ref{sec:Notation and preliminaries}. First we derive exact results in Section \ref{sec: Exact analysis}, where we find the joint transform of the number of jobs in the system and the shot-noise rate, and derive expressions for the expected value and variance. Subsequently, in Section \ref{sec:asymptotic analysis}, we derive a functional central limit theorem for this model.

\label{sec:A single infinite-server queue}
\subsection{Exact analysis}
\label{sec: Exact analysis}
We let $J_i$ be the service requirement of the $i$-th job, where  $J_1,J_2,\ldots$ are assumed to be i.i.d.; in the sequel $J$ denotes a random variable that is equal in distribution to $J_1$. Our first objective is to find the distribution of the number of jobs in the system at time $t$, in the sequel denoted by $N(t)$. This can be found in several ways; because of the appealing underlying intuition, we here provide an argument in which we approximate the arrival rate on intervals of length $\Delta$ by a constant, and then let $\Delta\downarrow 0$.

This procedure works as follows. 
We let $\L(t)=\Lambda(\omega,t)$ be an arbitrary sample path of the driving shot-noise process. Given $\Lambda(t)$, the number of jobs that arrived in the interval $[k\Delta,(k+1)\Delta)$ and are still in the system at time $t$, has a Poisson distribution with parameter $\Pb(J>t-(k\Delta+\Delta U_k))\cdot \Delta\Lambda(k\Delta) + o(\Delta)$, where $U_1, U_2,\ldots$ are i.i.d.\ standard uniform random variables. Summing over $k$ yields that the number of jobs in the system at time $t$ has a Poisson distribution with parameter 
\[
\sum_{k=0}^{t/\Delta-1} \Pb(J>t-(k\Delta+\Delta U_k)) \Delta\Lambda(k\Delta) + o(\Delta),
\]
which converges, as $\D\downarrow0$, to
\begin{equation}
\label{eq:randompar}
\int_0^t \Pb(J>t-u) \L(u) \dif u.
\end{equation}
The argument above is not new: a similar observation was mentioned in e.g.\ \cite{EMW1993}, for deterministic rate functions. Since $\L(\cdot)$ is actually a stochastic process, we conclude that the number of jobs has a mixed Poisson distribution, with the expression in Eqn.\ \eqref{eq:randompar} as random parameter. 
As a consequence, we find by conditioning on $\mathcal{F}_t$, 
\begin{eqnarray}
\label{trans1}
\nonumber\xi(t,z,s)&:=&\E z^{N(t)} e^{-s\Lambda(t)} = \E\left(e^{-s\L(t)}\E\left(z^{N(t)}\,|\,\mathcal{F}_t\right)\right)\\
&=& \E\exp\left( \int_0^t (z-1) {\mathbb P}(J>t-u)\Lambda(u) \dif u - s\Lambda(t) \right).
\end{eqnarray}
We have found the following result.
\begin{theorem}\label{THM1}
Let $\L(\cdot)$ be a shot-noise process. Then
\begin{equation}
\label{eq:thm}
\log \xi(t,z,s)= \nu \int_0^t\left( \beta\left((1-z) e^{rv}
\int_{v}^t {\mathbb P}(J>t-u)e^{-ru}\dif u + s e^{-r(t-v)} \right) -1\right)\dif v.
\end{equation}
\end{theorem}
\begin{proof}
The result follows directly from Lemma \ref{lemma} and Eqn.\ \eqref{trans1}.
\end{proof}

In Thm.\ \ref{THM1} we found that $N(t)$ has a Poisson distribution with the random parameter given in Eqn.\ \eqref{eq:randompar}. This leads to the following expression for the expected value 
\begin{equation}
\label{eq:EN}
\E N(t) = \int_0^t \E \Lambda(u) \Pb(J>t-u) \dif u.
\end{equation}
In addition, by the law of total variance we find
\begin{equation}\label{eq:VN}
\Var N(t) = \Var \left(\int_0^t \L(u) \Pb(J>t-u) \dif u\right) + \E \left(\int_0^t \L(u) \Pb(J>t-u) \dif u\right).
\end{equation}
The latter expression we can further evaluate, using an approximation argument that resembles the one we used above. 
Using a Riemann sum approximation, we find
\begin{eqnarray*}
\lefteqn{\Var \left(\int_0^t \Lambda(u) \Pb(J>t-u) \dif u\right) = \lim_{\D\downarrow0} \Var \left(\sum_{i=0}^{t/\D-1} \L(i \D) \Pb(J> t-i \D) \D\right)}\\
&=&2 \lim_{\D\downarrow0} \sum_{i=0}^{t/\D-1} \sum_{j>i}^{t/\D-1} \Cov(\L(i \D) \Pb(J>t-i \D) \D, \L(j \D) \Pb(J>t-j \D) \D)\\
&=& 2\int_0^t \int_v^t \Cov(\L(u),\L(v)) \Pb(J>t-u) \Pb(J>t-v) \dif u \dif v.
\end{eqnarray*}

Assuming that $u\geq v$, we know that $\Cov(\L(u),\L(v)) = e^{-r(u-v)} \Var \L(v)$ (cf. Lemma \ref{cov}). 
We thus have that (\ref{eq:VN}) equals
\[2\int_0^t \int_v^t  e^{-r(u-v)} \Var \L(v) \Pb(J>t-u) \Pb(J>t-v) \dif u \dif v +  \int_0^t  \E\L(u) \Pb(J>t-u) \dif u.
\]
We can make this more explicit using the corresponding formulas in \eqref{eq:momentsSN}.

\begin{example}[Exponential case]\label{Ex32}
Consider the case in which $J$ is exponentially distributed with mean $1/\mu$ and $\L(0)=0$. Then we can calculate the mean and variance explicitly. For $\mu\neq r$,
\[
\E N(t) = \f{\E \Lambda}{\mu} h_{r,\mu}(t),
\]
where the function $h_{r,\mu}(\cdot)$ is defined by
\[
t\mapsto \left\{ 
\begin{array}{ll} 
{\displaystyle \f{\mu(1-e^{-rt})-r(1-e^{-\mu t})}{\mu - r} }	& \text{if } \mu\neq r,\\
1-e^{-rt}-rte^{-rt} 				& \text{if } \mu = r.
\end{array}
\right.
\]
For the variance, we thus find for $\mu\neq r$
\[
\Var N(t) = \f{\nu\E B^2}{2r}\f{r^2(1-e^{-2\m t})+\m^2(1-e^{-2rt}) + \m r(4e^{-t(\m+r)}-e^{-2\m t} - e^{-2rt} -2)}{\m (\m-r)^2 (\m+r)}+\E N(t),
\]
and for $\mu=r$ 
\[
\Var N(t) = \f{\nu\E B^2}{4 r^3}{\left(1-e^{-2rt}-2rt(1+rt)e^{-2rt}\right)}{} + \E N(t).
\]
\end{example}

\subsection{Asymptotic analysis}
\label{sec:asymptotic analysis}

This subsection focuses on deriving a functional central limit theorem (FCLT) for the model under study, after having appropriately scaled the shot rate of the shot-noise process. 
In the following we assume that the service requirements are exponentially distributed with rate $\mu$, and we point out how it can be generalized to a general distribution in Remark \ref{rem:kiefer} below. We follow the standard approach to derive the FCLT for infinite-server queueing systems; we mimic the argumentation used in e.g.\
\cite{PW2007,Anderson2016}. As the proof has a relatively large number of standard elements, we restrict ourselves to the most important steps.

\vb

We apply a linear scaling to the shot rate of the shot-noise process, i.e.\ $\nu\mapsto n\nu$. It is readily checked that under this scaling, the steady-state level of the shot-noise process, as well as the steady-state number of jobs in the queue, blow up by a factor $n$. It is our objective to prove that, after appropriate centering and normalization, the process recording the number of jobs in the system converges to a Gaussian process. 

In the $n$-th scaled model, the number of jobs in the system at time $t$, denoted by $N_n(t)$, has the following (obvious) representation: with $A_n(t)$ denoting the number of arrivals in $[0,t]$, and $D_n(t)$ the number of departures,
\begin{equation}
\label{eq:FCLTrep}
N_n(t) = N_n(0) + A_n(t) - D_n(t).
\end{equation}
Here, $A_n(t)$ corresponds to a Cox process with a shot-noise-driven rate,  and therefore we have, with $\Lambda_n(s)$ the shot-noise in the scaled model at time $s$ and $S_A(\cdot)$ a unit-rate Poisson process,
\[A_n(t) = S_A\left(\int_0^t \Lambda_n(u){\rm d}u\right);\]
in line with our previous assumptions, we put $\L_n(0)=0.$
For our infinite-server model the departures $D_n(t)$ can be written as, 
with $S_D(\cdot)$  a unit-rate Poisson process (independent of $S_A(\cdot)$),
\[D_n(t)=S_D\left(\int_0^t \mu N_n(u)\dif u\right).\]

We start by identifying the average behavior of the process $N_n(t)$. 
Following the reasoning of \cite{Anderson2016}, assuming that $N_n(0)/n \Rightarrow \rho(0)$ (where `$\Rightarrow$' denotes weak convergence),
$N_n(t)/n$ converges almost surely to the solution of
\begin{equation}\label{RHO}
\rho(t) = \rho(0) + \int_0^t \E\L(u)\dif u - \int_0^t \m \rho(u)\dif u.
\end{equation}
\iffalse
Let $\Rightarrow$ and $\stackrel\Pb\longrightarrow$ refer to weak convergence and convergence in probability, respectively, and define the function space $D=D([0,\infty),\R)$ as the set of all right-continuous functions with left limits, which we endow with the standard Skorohod $J_1$ topology, see e.g.\ \cite{Billingsley}. The following uniform convergence can be established:
\[
\sup_{0\leq s\leq t} \left|\f1n A_n(t) - \int_0^t \E\L(u) \dif u\right| \stackrel{\Pb}{\longrightarrow}0.
\]
Furthermore, assume that $\f1n N_n(0) \Rightarrow \rho(0)$ in $\R$. Then $\f1n N_n(t)$ converges in probability, uniformly on compacts, to the solution of
$
\rho(t) = \rho(0) + \int_0^t \E\L(u)\dif u - \int_0^t \m \rho(s)\dif s.
$ Assuming $\E\Lambda(0)=0$,\fi
This equation is solved by $\rho(t)= \E N(t)$, with $\E N(t)$ provided in Example \ref{Ex32}.

Now we move to the derivation of the FCLT. Following the approach used in \cite{Anderson2016}, we proceed by studying an FCLT for the input rate process. 
To this end, we first define
\[\hat{\L}_n(t) := \sq n \left(\f1n\L_n(t)-\E\L(t)\right);
\qquad \hat{K}_n(t) := \int_0^t \hat\L_n(u)\dif u.\]
%We define the function space $D=D([0,\infty),\R)$ as the set of all right-continuous functions with left limits, which we endow with the standard Skorohod $J_1$ topology, see e.g.\ \cite{Billingsley}.

The following lemma states that $\hat K_n(\cdot)$ converges to an integrated Ornstein-Uhlenbeck (OU) process, corresponding to an OU process $\hat{\L}(\cdot)$ with a speed of mean reversion equal to $r$, long-run equilibrium level 0, and variance $\s_\L^2:=\nu\E B^2/(2r)$. 

\begin{lemma}
\label{lem:IOU}
Assume that for the shot sizes, distributed as $B$, it holds that $\E B,\E B^2<\infty$. Then $\hat{K}_n(\cdot)\Rightarrow \hat{K}(\cdot)$ as $n\to\infty$, where
\begin{equation}
\label{eq:barL}
\hat{K}(t) = \int_0^t \hat{\L}(u) \dif u,
\end{equation}
in which $\hat{\L}$ satisfies, with $W_1(\cdot)$  a standard Brownian motion,
\begin{equation}
\label{eq:barhat}
\hat{\L}(t) = \s_\Lambda W_1(t)- r \int_0^t \hat{\L}(u)\dif u.
\end{equation}
\end{lemma}
\begin{proof}
This proof is standard; for instance from \cite[Prop. 3]{Britt}, by putting the $\l_d$ in that paper to zero, it follows that $\hat\L_n(\cdot) \Rightarrow \hat\L(\cdot)$. This implies   $\hat{K}_n(\cdot)\Rightarrow \hat{K}(\cdot)$, using  \eqref{eq:barL} together with the continuous mapping theorem.
\end{proof}

Interestingly, the above result entails that the arrival rate process displays mean-reverting behavior. This also holds for the job count process in standard infinite-server queues. In other words, the job count process in the queueing system we are studying, can be considered as the composition of two mean-reverting processes. We make this more precise in the following.

\vb

From now on we consider the following centered and normalized version of the number of jobs in the system:
\[\hat N_n(t) := \sq n \left(\f1n N_n(t) - \rho(t)\right).\]
We assume that $\hat N_n(0)\Rightarrow \hat N(0)$ as $n\to\infty$.
To prove the FCLT, we rewrite $\hat N_n(t)$ in a convenient form. Mimicking the steps performed in  \cite{Anderson2016} or \cite{PW2007}, with $\bar S_A(t):=S_A(t)-t$, $\bar S_D(t):=S_D(t)-t$,
\[R_n(t):= \bar S_A\left(\int_0^t \Lambda_n(u){\rm d}u\right) - \bar S_D\left(\mu\int_0^t N_n(u){\rm d}u\right),\]
and using the relation (\ref{RHO}), we eventually obtain
\[\hat N_n(t)=\hat N_n(0) +\frac{R_n(t)}{\sqrt{n}} +\hat K_n(t)-\mu\int_0^t \hat N_n(u){\rm d}u.\]
Our next goal is to apply the martingale FCLT to the martingales $R_n(t)/\sqrt{n}$; see for background on the martingale FCLT for instance \cite{Ethier} and \cite{whitt2007}. The quadratic variation
equals
\[\left[\frac{R_n}{\sqrt{n}}\right]_t = \frac{1}{n}\left(S_A\left(\int_0^t \Lambda_n(u){\rm d}u\right) + S_D\left(\mu\int_0^t N_n(u){\rm d}u\right)\right),\]
which converges to $\int_0^t \E\L(u){\rm d}u+\mu\int_0^t\rho(u){\rm d}u.$
Appealing to the martingale FCLT, the following FCLT is obtained.  
\begin{theorem}
\label{thm:FCLT}
The centered and normalized version of the number of jobs in the queue satisfies an FCLT:  $\hat N_n(\cdot)\Rightarrow\hat N(\cdot)$ as $n\to\infty$, where $\hat N(t)$  solves the stochastic integral equation
\[\hat N(t) =\hat N(0) +\int_0^t \sqrt{\E\L(u)+\mu \rho(u)} \,{\rm d}W_2(u) + \hat K(t) 
-\mu\int_0^t \hat N(u){\rm d}u,\]
with $W_2(\cdot)$ a standard Brownian motion that is independent of the Brownian motion $W_1(\cdot)$ we introduced in the definition of $\hat K(\cdot)$.
\end{theorem}

\begin{remark}
\label{rem:arrival}
In passing, we have proven that the arrival process as such obeys an FCLT. With
\[\hat{A}_n(t) 	:= \sq n \left(\f1n A_n(t)-\int_0^t \E\L(u)\dif u\right),\]
we find that  $\hat{A}_n(t)\Rightarrow \hat A(t)$ as $n\to\infty$, where
\[\hat A(t) := \int_0^t \sqrt{\E\L(u)+\mu \rho(u)} \,{\rm d}W_2(u) + \hat K(t)=\int_0^t \sqrt{2\mu \rho(u)+\rho'(u)} \,{\rm d}W_2(u) + \hat K(t);\]
the last equality follows from the fact that $\rho(\cdot)$ satisfies (\ref{RHO}).
\end{remark}

\begin{remark}
\label{rem:kiefer}
The FCLT can be extended to non-exponential service requirements, by making use of  \cite[Thm.\ 3.2]{PW2010}. Their approach relies on two assumptions:
\begin{itemize}
\item[$\circ$] The arrival process should satisfy an FCLT;
\item[$\circ$] The service times are i.i.d.\ non-negative random variables with a general c.d.f.\, independent of the arrival process.
\end{itemize}
As noted in Remark \ref{rem:arrival}, the first assumption is satisfied for the model in this paper. The second assumption holds as well. In the non-exponential case the results are less clean; in general, the limiting process can be expressed in terms of a Kiefer process, cf.\ e.g.\ \cite{strongapproximations}.
\end{remark}

\section{Networks}
\label{sec:Networks}
Now that the reader is familiar with the one-dimensional setting, we extend this to networks. In this section, we focus on feedforward networks in which each node corresponds to an infinite-server queue. Feedforward networks are defined as follows.

\begin{definition}[feedforward network]
\label{def:ff}
Let $G=(V,E)$ be a directed graph with nodes $V$ and edges $E$. The nodes represent infinite-server queues and the directed edges between the facilities demonstrate how jobs move through the system. We suppose that there are no cycles in $G$, i.e. there is no sequence of nodes, starting and ending at the same node, with each two consecutive nodes adjacent to each other in the graph, consistent with the orientation of the edges.
\end{definition}

We focus on feedforward networks to keep the notation manageable. In Thm.\ \ref{thm:main}, we derive the transform of the numbers of jobs in all nodes, jointly with the shot-noise process(es) for feedforward networks. Nonetheless, we provide Example \ref{example: network with loops}, to show that analysis is in fact possible if there is a loop, but at the expense of more involved calculations. 

\vb

Since all nodes represent infinite-server queues, one can see that whenever a node has multiple input streams, it is equivalent to multiple infinite-server queues that work independently from each other, but have the same service speed and induce the same service requirement for arriving jobs. Consider Fig.\ \ref{fig:graaf} for an illustration. The reason why this holds is that different job streams move independently through the system, without creating waiting times for others. Therefore, merging streams do not increase the complexity of our network. The same holds for `splits' in job streams. By this we mean that after jobs finished their service in a server, they move to server $i$ with probability $q_i$ (with $\sum_i q_i=1$). Then, one can simply sample the entire path that the job will take through the system, at the arrival instance at its first server. 

\begin{figure}
\centering
\begin{tikzpicture}[>=latex ,auto ,node distance =4 cm and 5cm ,on grid ,
semithick ,
state/.style ={ circle , color =white ,
draw,processblue , text=black , minimum width =1 cm}]
\node[state] (A)
{$1$};
\node[state] (B) [right =2 of A] {$3$};
\node[state] (C) [below =2 of A] {$2$};
\node[state] (D) [below =2 of B] {$3'$};
\node[state] (E) [below left=1 and 5.2 of A]{$3$};
\node[state] (F) [below left=2 of E]{$2$};
\node[state] (G) [above left=2 of E]{$1$};
\path[->] (A) edge (B);
\path[->] (C) edge (D);
\path[->] (F) edge (E);
\path[->] (G) edge (E);
\coordinate[left =1.5 of G](c);
\coordinate[left =1.5 of F](d);
\coordinate[right=1.5 of E](e);
\draw[->] (c) to [left] node[auto] {} (G);
\draw[->] (d) to [left] node[auto] {} (F);
\draw[->] (E) to [right] node[auto] {} (e);
\coordinate[right =1.5 of D](f);
\coordinate[right =1.5 of B](g);
\draw[->] (B) to [right] node[auto] {} (g);
\draw[->] (D) to [right] node[auto] {} (f);
\coordinate[left = 1.5 of A](h);
\coordinate[left = 1.5 of C](i);
\draw[->] (h) to [right] node[auto] {} (A);
\draw[->] (i) to [right] node[auto] {} (C);
\end{tikzpicture}
\caption{\textit{Since the jobs are not interfering with each other, the network on the left is equivalent to the graph on the right. Node 3' is a copy of node 3: it works at the same speed and induces the same service requirements.}}
\label{fig:graaf}
\end{figure}
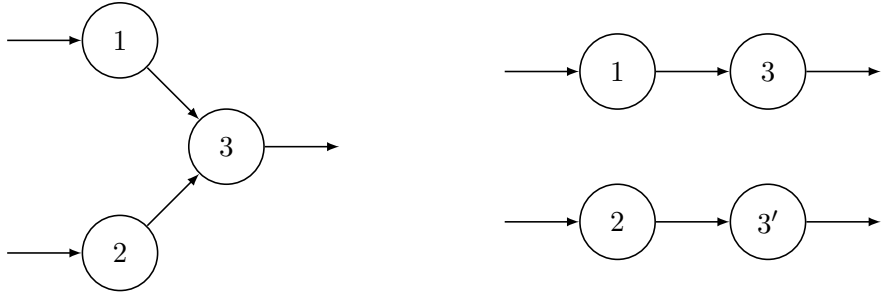

If one recognizes the above, then all feedforward networks reduce to parallel tandem systems in which the first node in each tandem system is fed by external input. The procedure to decompose a network into parallel tandems consists of finding all paths between nodes in which jobs either enter or leave the system. Each of these paths will subsequently be considered as a tandem queue, which are then set in parallel. 

To build up to the main result, we first study tandem systems in Section \ref{sec:tandem}. Subsequently, we put the tandem systems in parallel in Section \ref{sec:parallel} and finally we present the main theorem and some implications in Section \ref{sec:mainthm}. 

\subsection{Tandem systems}
\label{sec:tandem}
As announced, we proceed by studying tandem systems. In Section \ref{sec:parallel} below, we study $d$ parallel tandem systems, where $i=1,\ldots,d$. In this subsection we consider the $i$-th of these tandem systems.

Suppose that tandem $i$ has $S_i$ service facilities and the input process at the first node is Poisson, with a shot-noise arrival rate $\L_i(\cdot)$. We assume that jobs enter node $i1$. When they finish service, they enter node $i2$, etc., until they enter node $iS_i$ after which they leave the system. We use $ij$ as a subscript referring to node $j$ in tandem system $i$ and we refer to the node as node $ij$. Hence $N_{ij}(t)$ and $ J_{ij}$ denote the number of jobs in node $ij$ at time $t$, and a copy of a service requirement, respectively, where $j=1,\ldots,S_i$. 

Fix some time $t>0$. Again we derive results by splitting time into intervals of length $\Delta$. Denote by $M_{ij}(k,\D)$ the number of jobs present in node $ij$ at time $t$ that have entered node $i1$ between time $k\D$ and $(k+1)\D$; as we keep $t$ fixed we suppress it in our notation. 
Because jobs are not interfering with each other in the infinite-server realm, we can decompose the transform of interest:
\begin{equation}
\label{eq:chop}
\E \left(\prod_{j=1}^{S_i} z_{ij}^{N_{ij}(t)}\right) = \lim_{\D\downarrow0}\prod_{k=0}^{t/\D-1} \E \left(\prod_{j=1}^{S_i} z_{ij}^{M_{ij}(k,\D)}\right).
\end{equation}
Supposing that the arrival rate is a deterministic function of time $\l_i(\cdot)$,
by conditioning on the number of arrivals in the $k$-th interval,
\begin{align*}
 \E \left(\prod_{j=1}^{S_i} z_{ij}^{M_{ij}(k,\Delta)} \right)&= \sum_{m=0}^\infty e^{-\l_i(k\Delta) \Delta} \f{(\l_i(k\Delta)\Delta)^m}{m!} \left( f_i(k\Delta,z) \right)^m\\
&= \exp\Big(\Delta \l_i(k\Delta)(f_i(k\Delta,z)-1) \Big),
\end{align*}
where
\begin{equation}
\label{eq:defpsi}
f_i(u,z) := p_i(u) + \sum_{j=1}^{S_i} z_{ij} p_{ij}(u),
\end{equation}
in which $p_i(u)$ ($p_{ij}(u)$, respectively) denotes the probability that the job that entered tandem $i$ at time $u$ has already left the tandem (is in node $j$, respectively) at time $t$.
Note that
\begin{equation}
\label{eq:p}
p_i(u)		= \Pb\left(\sum_{\ell=1}^{S_i} J_{i\ell} < t-u\right),\quad
p_{ij}(u) 	= \Pb\left(\sum_{\ell=1}^{j-1} J_{i\ell} < t-u, \sum_{\ell=1}^{j} J_{i\ell}>t-u\right).
\end{equation}
Recognizing a Riemann sum and letting $\D\downarrow0$, we conclude that Eqn.\ \eqref{eq:chop} takes the following form:
\[
\E \left(\prod_{j=1}^{S_i} z_{ij}^{N_{ij}(t)} \right)= \exp\left(\int_0^t \lambda_i(u) (f_i(u,z)-1)\dif u\right).
\]
In case of a stochastic rate process $\Lambda_i(\cdot)$, we obtain
\begin{equation*}
\E\left.\left(\prod_{j=1}^{S_i} z_{ij}^{N_{ij}(t)} \,\right|\,\L_i(\cdot)\right) = \exp\left(\int_0^t \Lambda_i(u) (f_i(u,z)-1)\dif u\right).
\end{equation*}
Therefore it holds that
\begin{eqnarray*}
\E\left( \prod_{j=1}^{S_i} z_{ij}^{N_{ij}(t)} e^{-s\L_i(t)}\right)& =& \E\left(\E\left(\left.\prod_{j=1}^{S_i} z_{ij}^{N_{ij}(t)}e^{-s\L_i(t)}\,\right|\,\L_i(\cdot)\right)\right)\\& =& \E\left(e^{-s\L_i(t)} \E\left(\left.\prod_{j=1}^{S_i} z_{ij}^{N_{ij}(t)}\,\right|\,\L_i(\cdot)\right)\right),
\end{eqnarray*}
and we consequently find
\begin{equation}
\label{eq:trans2}
\E\left( \prod_{j=1}^{S_i} z_{ij}^{N_{ij}(t)} e^{-s\L_i(t)}\right)  = \E \exp\left(\int_0^t \Lambda_i(u) (f_i(u,z)-1)\dif u - s\L_i(t)\right).
\end{equation}

\subsection{Parallel (tandem) systems}
\label{sec:parallel}
Now that the tandem case has been analyzed, the next step is to put the tandem systems as described in Section \ref{sec:tandem} in parallel. We assume that there are $d$ parallel tandems. There are different ways in which dependence between the parallel systems can be created. Two relevant models are listed below, and illustrated in Fig.\ \ref{fig:M1M2}.

\begin{itemize}
\item[\textbf{(M1)}] Let $\L\equiv\L(\cdot)$  be a $d$-dimensional shot-noise process $(\L_1,\ldots,\L_d)$ where the shots in all $\L_i$ occur simultaneously (the shot distributions and decay rates may be different). The process $\L_i$, for $i=1,\ldots,d$, corresponds to the arrival rate of tandem system $i$. Each tandem system has an arrival process, in which the Cox processes are independent given their shot-noise arrival rates. 

\item[\textbf{(M2)}] Let $\L\equiv\L(\cdot)$ be the shot-noise rate of a Cox process. The corresponding Poisson process generates {\it simultaneous} arrivals in all tandems. 
\end{itemize}

\begin{figure}[h!]
\centering
\begin{tikzpicture}[>=latex ,auto ,node distance =4 cm and 5cm]
\node (L1) {$\Lambda_1$};
\node (L2) [below =0.8 of L1] {$\Lambda_2$};
\draw[<->] (L1) to (L2);
\node[box] (T1) [right = 1.2 of L1] {Tandem 1};
\path[->] (L1) edge (T1);
\node[box] (T2) [right = 1.2 of L2] {Tandem 2};
\path[->] (L2) edge (T2);
\coordinate[right =1.2 of T2](c2);
\draw[->] (T2) to [right] node[auto] {} (c2);
\coordinate[right =1.2 of T1](c1);
\draw[->] (T1) to [right] node[auto] {} (c1);
\node[box,right=4 of c1](T12){Tandem 1};
\coordinate[right =1.2 of T12](d1);
\draw[->] (T12) to [right] node[auto] {} (d1);
\node[box,right=4 of c2](T22){Tandem 2};
\coordinate[right =1.2 of T22](d2);
\draw[->] (T22) to [right] node[auto] {} (d2);
\node (L) [below right = 0.4 and 1 of c1] {$\Lambda$};
%\coordinate[right=0.5 of L](split);
\node[phase] (split) [right=1 of L] {};
\draw[->] (L) to (split);
\coordinate[left=0 of T12] (e1);
\coordinate[left=0 of T22] (e2);
\draw[->] (split) to (e1);
\draw[->] (split) to (e2);
\end{tikzpicture}
\caption{\textit{Model (M1) is illustrated on the left, and Model (M2) is illustrated on the right. The rectangles represent tandem systems, which consist of an arbitrary number of nodes in series.}}
\label{fig:M1M2}
\end{figure}
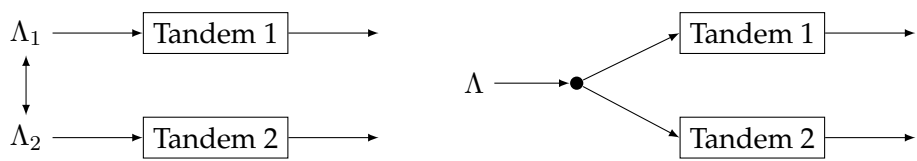

\begin{remark}
The model in which there is essentially one shot-noise process that generates arrivals for all queues independently, is a special case of Model (M1). This can be seen by setting all components of $\L=(\L_1,\ldots,\L_d)$ equal, by letting the shots and decay rate be identical.
\end{remark}
In Model (M1), correlation between the shot-noise arrival rates induces correlation between the numbers of jobs in the different queues. In Model (M2), correlation clearly appears because all tandem systems have the same input process. Of course, the tandem systems will not behave identically because the jobs may have different service requirements. In short, correlation across different tandems in Model (M1) is due to linked arrival rates, and correlation in Model (M2) is due to simultaneous arrival epochs. We feel that both versions are relevant, depending on the application, and hence we analyze both.

\paragraph{Analysis of (M1) ---}
Suppose that the dependency is of the type as in Model (M1). This means that the shots, in each component of $\L$, occur simultaneously. Recall the definition of $f_i$ as stated in Eqn.\ \eqref{eq:defpsi}. It holds that
\begin{align}
\label{eq:prelemmaM1}
\nonumber \E\left(\prod_{i=1}^d e^{-s_i\L_i(t)} \prod_{j=1}^{S_i} z_{ij}^{N_{ij}(t)}\right)
\nonumber &= \E\left(\prod_{i=1}^d \E\left(\left.\prod_{j=1}^{S_i} z_{ij}^{N_{ij}(t)} e^{-s_i\L_i(t)} \,\right|\, \L_i(\cdot)\right)\right)\\
\nonumber &= \E\left(\prod_{i=1}^d\left( e^{-s_i\L_i(t)} \E\left(\left.\prod_{j=1}^{S_i} z_{ij}^{N_{ij}(t)}\, \right| \,\L_i(\cdot)\right)\right)\right)\\
&=\E\exp\left(\sum_{i=1}^d\left(\int_0^t \L_i(u)(f_i(u,z)-1)\dif u - s_i\L_i(t)\right)\right),
\end{align}
where the last equality holds due to (\ref{eq:trans2}).

\paragraph{Analysis of (M2) ---}
Now suppose that the dependency in this model is of type (M2), i.e., there is one shot-noise process that generates simultaneous arrivals in the parallel tandem systems. 
%Then we follow a similar approach as in the parallel (M2)-system, but now one has to keep track of in which of the nodes in tandem the job is located. 

First we assume a deterministic arrival rate function $\l(\cdot)$. Let $M_{ij}(k,\D)$ be the number of jobs present in tandem system $i$ at node $j$ at time $t$ that have arrived in the system between $k\D$ and $(k+1)\D$. Note that
\[
\E\left(\prod_{i=1}^d \prod_{j=1}^{S_i} z_{ij}^{N_{ij}(t)}\right)
= \lim_{\D\downarrow0} \prod_{k=0}^{t/\D-1} \E\left(\prod_{i=1}^d\prod_{j=1}^{S_i} z_{ij}^{M_{ij}(k,\D)}\right).
\]
To further evaluate the right hand side of the previous display, we observe that we can write
\[
\E\left(\prod_{i=1}^d \prod_{j=1}^{S_i} z_{ij}^{M_{ij}(k,\D)}\right)
= \sum_{m=0}^\infty e^{-\l(k\D)\D}\f{\l(k\D)\D)^m}{m!}(f(k\D,z))^m
=e^{\D\l(k\D)(f(k\D,z)-1)},\]where
\begin{equation}
\label{eq:defupsilon}
f(u,z) := \sum_{j=1}^d \sum_{\ell_j=1}^{S_{j+1}} p_{\ell_1,\ldots, \ell_d} \prod_{i=1}^d z_{i\ell_i};
\end{equation}
in this definition $p_{\ell_1,\ldots, \ell_d}\equiv p_{\ell_1,\ldots, \ell_d}(u)$ equals the probability that a job that arrived at time $u$ in tandem $i$ is in node $\ell_i$ at time $t$ (cf.\ Eqn.\ \eqref{eq:p}). The situation that $\ell_i=S_i+1$ means that the job left the tandem system; we define $z_{i,S_i+1} = 1$. 

In a similar fashion as before, we conclude that
\begin{equation}
\label{eq:prelemmaM2}
\E\left( \prod_{i=1}^d \prod_{j=1}^{S_i} z_{ij}^{N_{ij}(t)} e^{-s \L(t)} \right)= \E\exp\left(\int_0^t \L(u)(f(u,z)-1)\dif u - s\L(t)\right),
\end{equation}
with $f$ defined in Eqn.\ \eqref{eq:defupsilon}.

\begin{example}[Two-node parallel system]
\label{example:two-node parallel system}
In the case of a parallel system of two infinite-server queues, $f(u,z)$ simplifies to
\begin{align*}
f(u,z_{11},z_{21}) = \sum_{\ell_1=1}^2 \sum_{\ell_2=1}^2 z_{1\ell_1}z_{2\ell_2} p_{\ell_1,\ell_2} = z_{11}z_{21} p_{11} + z_{21}p_{21} + z_{11}p_{12} + p_{22}.
\end{align*}
\end{example}

\begin{remark}[Routing]
Consider a feedforward network with routing. As argued in the beginning of this section, the network can be decomposed as a parallel tandem system. In case there is splitting at some point, then one decomposes the network as a parallel system, in which each tandem $i$ receives the job with probability $q_i$, such that $\sum q_i=1$. This can be incorporated simply by adjusting the probabilities contained in $f_i$ in Eqn.\ \eqref{eq:prelemmaM1}, which are given in Eqn.\ \eqref{eq:p}, so that they include the event that the job joined the tandem under consideration. For instance, the expression for $p_i(u)$ in the left equation in \eqref{eq:p} would become
\[
\Pb\bigg(Q=i, \sum_{\ell=1}^{S_i} J_{i\ell}<t-u\bigg),
\]
where $Q$ is a random variable with a generalized Bernoulli (also called `categorical') distribution, where
\[
\Pb(\text{job is assigned to tandem } i) = \Pb(Q=i)=q_i,\quad\text{for } i=1,\ldots d,
\] 
with $\sum q_i=1$; the right equation in \eqref{eq:p} is adjusted similarly. Other than that, the analysis is the same for the case of splits.
\end{remark}

\begin{remark}[Networks with loops]
\label{example: network with loops}
So far we only considered feedforward networks. 
Networks with loops can be analyzed as well, but the notation becomes quite cumbersome. To show the method in which networks with loops and routing can be analyzed, we consider a specific example. Suppose that arrivals enter node one, after which they enter node two. After they have been served in node two, they go back to node one with probability $\eta$, or leave the system with probability $1-\eta$. In this case, with similar techniques as before, we can find
\[
\E z_1^{N_1(t)}z_2^{N_2(t)} = \exp\left(\int_0^t \L(u)(f(u,z_1,z_2)-1)\dif u \right),
\]
with 
\[
f(u,z_1,z_2) = \Pb(\,\text{job($u$) left system}) + \sum_{i=1}^2 z_i \Pb(\,\text{job($u$) is in node $i$}),
\]
in which $\text{job}(u)$ is the job that arrived at time $u$ and we are examining the system at time $t$. Now, if we denote service times in the $j$-th node by $J^{(j)}$, then, at a specific time $t$,
\[
\Pb(\,\text{job($u$) left system}) = \sum_{k=0}^\infty \Pb\left(\sum_{i=1}^{k+1} (J_i^{(1)}+J_i^{(2)}) \leq t-u \right) \eta^k(1-\eta).
\]
Analogously, $\Pb(\,\text{job($u$) is in node $1$})$ equals, by conditioning on the job having taken $k$ loops,
\[
\sum_{k=0}^\infty \eta^{k} \Pb\left(J_{k+1}^{(1)} + \sum_{i=1}^{k} (J_i^{(1)} + J_i^{(2)}) > t-u , \sum_{i=1}^{k} (J_i^{(1)} + J_i^{(2)}) \leq t-u\right);
\]
likewise, $\Pb(\,\text{job($u$) is in node $2$})$ equals
\[
\sum_{k=0}^\infty \eta^{k} \Pb\left(\sum_{i=1}^{k+1}( J_i^{(1)} + J_i^{(2)}) > t-u , J^{(1)}_{k+1} + \sum_{i=1}^{k}( J_i^{(1)} + J_i^{(2)}) \leq t-u\right).
\]
For example, in case all $J_i^{(j)}$ are independent and exponentially distributed with mean $1/\mu$, we can calculate those probabilities explicitly. Indeed, if we denote by $Y$ a Poisson process with rate $\mu$, then e.g.,
\begin{align*}
\Pb\left(\sum_{i=1}^{k+1} (J_i^{(1)} + J_i^{(2)}) > t-u , J^{(1)}_{k+1} +  \sum_{i=1}^{k} (J_i^{(1)} + J_i^{(2)} )\leq t-u\right) &= \Pb(Y(t-u)=2k+1)\\
&= e^{-\mu (t-u)} \f{(\mu (t-u))^{2k+1}}{(2k+1)!}
\end{align*}
and thus
\[
\Pb(\,\text{job($u$) is in node $2$}) = \sum_{m=0}^\infty \eta^m  e^{-\mu (t-u)} \f{(\mu (t-
u))^{2m+1}}{(2m+1)!}.
\]
A similar calculation can be done for the probability that the job is in node one. Recalling that a sum of exponentials has a Gamma distribution, we can write
\begin{eqnarray*}
f(u,z_1,z_2) &=&  z_1\sum_{m=0}^\infty \eta^m  e^{-\mu (t-u)} \f{(\mu (t-u))^{2m}}{(2m)!} + z_2 \sum_{m=0}^\infty \eta^m  e^{-\mu (t-u)} \f{(\mu (t-u))^{2m+1}}{(2m+1)!}\\
&&+ \,\sum_{m=0}^\infty \eta^m(1-\eta) F_{\G(2m+2,\m)}(t-u)\\
&=& z_1 e^{-\mu(t-u)} \cosh\left( {\mu\sq{\eta} (t-u)} \right) + z_2 \f{e^{-\mu(t-u)}}{\sq{\eta}}\sinh\left(\mu\sq{\eta}(t-u) \right)\\
&&\,+ (1-\eta)\sum_{
m=0}^\infty \eta^m F_{\G(2m+2,\m)}(t-u),
\end{eqnarray*}
where $F_{\G(2m+2, \m)}$ denotes the distribution function of a $\G$-distributed random variable with rate $\m$ and shape parameter $2m+2$.
\end{remark}

\subsection{Main result}
\label{sec:mainthm}
In this subsection we summarize and conclude with the following main result. Recall Definition \ref{def:ff} of a feedforward network. In the beginning of Section \ref{sec:Networks} we argued that we can decompose a feedforward network into parallel tandems. In Section \ref{sec:parallel} we studied exactly those systems, leading up to the following results.
\begin{theorem}
\label{thm:main}
Suppose we have a feedforward network of infinite-server queues, where the input process is a Poisson process with shot-noise arrival rate. Then the network can be decomposed into parallel tandem systems. In Model (M1), it holds that 
\begin{eqnarray*}
\lefteqn{\hspace{-35mm}\E \left(\prod_{i=1}^d \prod_{j=1}^{S_i} z_{ij}^{N_{ij}(t)} e^{-s\L_i(t)} \right)= \E\exp\left(\sum_{i=1}^d\left(\int_0^t \L_i(u)(f_i(u,z)-1)\dif u - s_i\L_i(t)\right)\right)}\\
&=&\exp\left(\nu \int_0^t \left(\beta(g(s,v))-1\right)\dif v\right),
\end{eqnarray*}
with $f_i(\cdot,\cdot)$ as defined in Eqn.\ \eqref{eq:defpsi} and where $g(s,v)$ is a vector-valued function in which component $i$ is given by
\[
s_i e^{-r_i(t-v)} - e^{r_iv} \int_v^t (f_i(u,z)-1) e^{-r_iu} \dif u.
\]
Furthermore, in Model (M2),
\begin{eqnarray*}
\lefteqn{\E\left( \prod_{i=1}^d \prod_{j=1}^{S_i} z_{ij}^{N_{ij}(t)}e^{-s\L(t)} \right)=\E\exp\left(\int_0^t \L(u)(f(u,z)-1)\dif u - s\L(t)\right)}\\
&=& \exp\left(\nu\int_0^t \left(\b\left( se^{-r(t-v)} - e^{rv} \int_v^t (f(u,z)-1)e^{-ru}\dif u \right)-1\right)\dif v\right),
\end{eqnarray*}
with $f(\cdot,\cdot)$ as defined in Eqn.\ \eqref{eq:defupsilon}.
\end{theorem}
\begin{proof}
These are Eqns.\ \eqref{eq:prelemmaM1} and \eqref{eq:prelemmaM2} to which we applied Lemma \ref{lemma}.
\end{proof}
Next we calculate covariances between nodes in tandem and parallel thereafter.
\paragraph{Covariance in Tandem System ---}Consider a tandem system consisting of two nodes and we want to analyze the covariance between the numbers of jobs in the nodes. Dropping the index of the tandem system, denote by $N_1(\cdot)$ and $N_2(\cdot)$ the numbers of jobs in node 1 and 2, respectively. Using Eqn.\ \eqref{eq:prelemmaM1}, differentiation yields 
\[
\E N_2(t) = \int_0^t \Pb(J_1<t-u, J_1+J_2>t-u) \E \L(u) \dif u
\]
and 
\[
\E N_1(t) N_2(t) = \E\left(\int_0^t \Pb(J_1<t-u, J_1+J_2>t-u) \L(u)\dif u \int_0^t \Pb(J_1>t-v)\L(v) \dif v\right)
\]
so that
\begin{eqnarray*}
\lefteqn{\Cov(N_1(t),N_2(t))}\\& =& \Cov\left(\int_0^t \Pb(J_1<t-u, J_1+J_2>t-u) \L(u)\dif u, \int_0^t \Pb(J_1>t-v)\L(v) \dif v\right)\\
&=& 2 \int_0^t \int_v^t \Pb(J_1<t-u, J_1+J_2>t-u)\Pb(J_1>t-v) \Cov(\L(u),\L(v)) \dif u \dif v\\
&= &2 \int_0^t \int_v^t \Pb(J_1<t-u, J_1+J_2>t-u)\Pb(J_1>t-v) e^{-r(u-v)} \Var \L(v) \dif u \dif v,
\end{eqnarray*}
cf.\ Eqn.\ \eqref{eq:momentsSN} for the last equality.

\paragraph{Covariance parallel (M1) ---}
Consider a parallel system consisting of two nodes only. In order to study covariance  in the parallel (M1) case, we need a result about the covariance of the corresponding shot-noise process.
\begin{lemma}
\label{cov}
Let $\L_1(\cdot),\L_2(\cdot)$ be shot-noise processes of which the jumps occur simultaneously according to a Poisson arrival process with rate $\nu$. Let the decay be exponential with rate $r_1, r_2$, respectively. Then it holds that, for $\d>0$,
\begin{equation}
\label{eq:covl1l2}
\Cov(\L_1(t),\L_2(t+\d)) = e^{-r_2\d} \Cov(\L_1(t),\L_2(t)) = e^{-r_2\d} \f{\nu \E B_{11} B_{12}}{r_1+r_2}(1-e^{-(r_1+r_2)t}),
\end{equation}
which, in case $\L_1=\L_2$, reduces to
\[
\Cov(\L_i(t),\L_i(t+\d)) = e^{-r_i\d} \Var \L_i(t),\quad\text{for } i=1,2,
\]
corresponding to {\em \cite[{\em p}.\ 394]{Ross}} and Eqn.\ \eqref{eq:momentsSN}.
\end{lemma}
\begin{proof}
See Appendix \ref{app:cov}.
\end{proof}
By making use of Eqn.\ \eqref{eq:prelemmaM1}, we find 
\[
\E N_1(t) N_2(t) = \E\left( \int_0^t \Lambda_1(u)\Pb(J_1>t-u)\dif u \int_0^t \Lambda_2(v)\Pb(J_2>t-v) \dif v\right). 
\]
This implies
\begin{eqnarray*}
\lefteqn{\Cov(N_1(t),N_2(t)) = \Cov\left(\int_0^t \L_1(u)\Pb(J_1>t-u)\dif u, \int_0^t \L_2(v)\Pb(J_2>t-v)\dif v\right)}\\
&=&2\int_0^t \int_v^t \Cov(\L_1(u), \L_2(v)) \Pb(J_1>t-u)\Pb(J_2>t-v) \dif u \dif v\\
&=&2\int_0^t \int_v^t \f{\nu \E B_{11}B_{12}}{r_1+r_2} \left(1-e^{-(r_1+r_2)v	}\right)e^{-r_2(u-v)}   \Pb(J_1>t-u)\Pb(J_2>t-v) \dif u \dif v
\end{eqnarray*}
where we made use of the fact that, for $u\geq v$,
\begin{align*}
\Cov(\L_1(u),\L_2(v)) = \f{\nu \E B_{11}B_{12}}{r_1+r_2} \left(1-e^{-(r_1+r_2)v}\right)e^{-r_2(u-v)},
\end{align*}
cf.\ Lemma \ref{cov}.

\paragraph{Covariance parallel (M2) ---}
Extracting the mixed moment from the transform in Eqn.\ \eqref{eq:prelemmaM2}, we derive directly that
\begin{eqnarray*}
\E N_1(t)N_2(t) &=& \E\left( \int_0^t \L(u) \Pb(J_1>t-u,J_2>t-u)\dif u\right)\\
&& +\, \E \left(\int_0^t \L(u) \Pb(J_1>t-u)\dif u \int_0^t \L(u) \Pb(J_2>t-u)\dif u\right).
\end{eqnarray*}
This implies
\begin{eqnarray*}
\Cov(N_1(t),N_2(t)) &=& \Cov\left(\int_0^t \L(u)\Pb(J_1>t-u)\dif u, \int_0^t \L(u)\Pb(J_2>t-u)\dif u\right)\\
&&+ \,\int_0^t \E \Lambda(u) \Pb(J_1>t-u,J_2>t-u) \dif u.
\end{eqnarray*}

The following proposition compares the correlations present in Model (M1) and (M2). In the proposition we refer to the number of jobs in queue $j$ in Model (M$i$) at time $t$ as $N_j^{(i)}(t)$, for $i=1,2$. We find the anticipated result, that the correlation in Model (M2) is stronger than in Model (M1).
\begin{proposition}
Let $\Lambda(\cdot)$ be the shot-noise process that generates simultaneous arrivals in both queues and let $\Lambda_1(\cdot), \Lambda_2(\cdot)$ be processes that have simultaneous jumps and generate arrivals in both queues independently. Suppose that $ \Lambda_1(t) \stackrel{{\rm d}}{=} \Lambda_2(t) \stackrel{{\rm d}}{=} \Lambda(t)$, for $t\geq 0$. Then, for any $t\ge 0$, 
\[
\Corr(N_1^{(1)}(t),N_2^{(1)}(t)) \leq \Corr(N_1^{(2)}(t),N_2^{(2)}(t)).
\]
\end{proposition}
\begin{proof}
Because of the assumption $ \Lambda_1(t) \stackrel{{\rm d}}{=} \Lambda_2(t) \stackrel{{\rm d}}{=} \Lambda(t)$, we have that, for all combinations $i,j\in\{1,2\}$, the $N_i^{(j)}(t)$ are equal in distribution. Therefore it is sufficient to show that
\[
\Cov(N_1^{(1)}(t), N_2^{(1)}(t)) \leq \Cov(N_1^{(2)}(t), N_2^{(2)}(t)).
\]
The expressions for the covariances, which are derived earlier in this section, imply that
\[\Cov(N_1^{(2)}(t), N_2^{(2)}(t)) - \Cov(N_1^{(1)}(t), N_2^{(1)}(t))= \E \int_0^t \Lambda(u) \Pb(J_1>t-u, J_2>t-u)\dif u,\]
which is non-negative, as desired.
\end{proof}

\section{Concluding remarks}
\label{sec:concluding remarks}
We have considered networks of infinite-server queues with shot-noise-driven Coxian input processes. For the single queue, we found explicit expressions for the Laplace transform of the joint distribution of the number of jobs and the driving shot-noise arrival rate, as well as a functional central limit theorem of the number of jobs in the system under a particular scaling. The results were then extended to a network context: we derived an expression for the joint transform of the numbers of jobs in the individual queues, jointly with the values of the driving shot-noise processes. 

We included the functional central limit theorem for the single queue, but it is anticipated that a similar setup carries over to the network context, albeit at the expense of considerably more involved notation. Our future research will include the study of the departure process of a single queue;
the output stream should remain Coxian, but of another type than the input process.

\subsection*{Acknowledgements}
The authors thank Marijn Jansen for insightful discussions about the functional central limit theorem in this paper. The research for this paper is partly funded by the NWO Gravitation Project NETWORKS, Grant Number 024.002.003. The research of Onno Boxma was also partly funded by the Belgian Government, via the IAP Bestcom Project. 

\begin{appendices}
\section{Proof of Lemma \ref{lemma}} \label{app:proof}
There are various ways to prove this result; we here include a procedure that intensively relies on the probabilistic properties of the shot-noise process involved.
Observe that, recognizing a Riemann sum,
\begin{equation}
\label{RS}
\int_0^t f(u,z)\Lambda(u)\,\dif u =\lim_{\Delta\downarrow 0} \Delta \sum_{k=1}^{t/\Delta} f(k\Delta,z)\Lambda(k\Delta).
\end{equation}
With $P_B(t)$ a Poisson process with rate $\nu$ and the $U_i$ i.i.d.\ samples from a uniform distribution on $[0,1]$, it holds that
\[\Lambda(k\Delta) = \sum_{\ell=1}^k \sum_{i = P_B({(\ell-1)\Delta})+1}^{P_B({\ell\Delta})} B_i e^{-r\Delta U_i}
e^{-r(k-\ell)\Delta}.\]
We thus obtain that the expression in (\ref{RS}) equals (where the equality follows by interchanging the order of the summations)
\begin{align*}
&\lim_{\Delta\downarrow 0} \Delta \sum_{k=1}^{t/\Delta}f( k\D,z)\sum_{\ell=1}^k \sum_{i = P_B({(\ell-1)\Delta})+1}^{P_B({\ell\Delta})} B_i  e^{-r\Delta U_i}
e^{-r(k-\ell)\Delta}\\
&=\,\lim_{\Delta\downarrow 0} \Delta \sum_{\ell=1}^{t/\Delta}\sum_{i = P_B({(\ell-1)\Delta})+1}^{P_B({\ell\Delta})}B_ie^{-r\Delta U_i} \sum_{k=\ell}^{t/\Delta}   f( k\D,z)
e^{-r(k-\ell)\Delta},
\end{align*}
which behaves as
\[\lim_{\Delta\downarrow 0} \sum_{\ell=1}^{t/\Delta}e^{r\ell\Delta}
\int_{\ell\Delta}^t f(u,z) e^{-ru}\dif u\sum_{i = P_B({(\ell-1)\Delta})+1}^{P_B({\ell\Delta})} B_i e^{-r\Delta U_i} .\]
Furthermore, we have the representation
\[
\Lambda(t) = \lim_{\Delta\downarrow0} \sum_{\ell=1}^{t/\Delta} \sum_{i=P((\ell-1)\Delta) +1}^{P_B(\ell\Delta)} B_i e^{-r(t-\ell\Delta + \Delta U_i)}.
\]
We conclude that $\E z^{N(t)} e^{-s\Lambda(t)}$ equals
\begin{equation}
\label{eq:v1}
\lim_{\Delta\downarrow 0}  {\mathbb E}\exp\left(
 \sum_{\ell=1}^{t/\Delta}\sum_{i = P_B({(\ell-1)\Delta})+1}^{P_B({\ell\Delta})} B_i \left(e^{-r\Delta U_i}e^{r\ell\Delta}
\int_{\ell\Delta}^t f(u,z) e^{-ru}\dif u - s   e^{-r(t-\ell\Delta+\Delta U_i)}\right)\right).
\end{equation}
Conditioning on the values of $P_B(\ell\Delta)-P_B((\ell-1)\Delta)$, for $\ell=1,\ldots, t/\Delta$, and using that the $B_i$ are i.i.d., we find that the expression in Eqn.\ \eqref{eq:v1} equals
\begin{align*}
&=\lim_{\Delta\downarrow 0 }  \prod_{\ell=1}^{t/\Delta} e^{-\nu \Delta} \sum_{k_\ell=0}^\infty \f{(\nu\Delta)^{k_\ell}}{k_{\ell}!} \left(\E\exp\left(B_1\left(e^{-r\Delta U_i}e^{r\ell\Delta} \int_{\ell\Delta}^t f(u,z) e^{-ru} \dif u - s e^{-r(t-\ell\Delta+\Delta U_i)}\right) \right)\right)^{k_\ell}\\
&=\lim_{\Delta\downarrow 0 } \prod_{\ell=1}^{t/\Delta} e^{-\nu \Delta} \exp \left(\nu\Delta \E\exp\left(B_1\left(e^{-r\Delta U_i}e^{r\ell\Delta} \int_{\ell\Delta}^t f(u,z) e^{-ru} \dif u - s e^{-r(t-\ell\Delta+\Delta U_i)}\right) \right)\right),
\end{align*}
which can be written as
\[\lim_{\Delta\downarrow 0}\exp\left(\nu\Delta \sum_{\ell=1}^{t/\Delta} \left(\beta\left( s e^{-r(t-\ell\Delta+\Delta U_i)} -e^{-r\Delta U_i}e^{r\ell\Delta}
\int_{\ell\Delta}^t f(u,z) e^{-ru}\dif u  \right) -1\right)\right).\]
The lemma now follows from continuity of the exponent and the definition of the Riemann integral.

\section{Proof of Lemma \ref{cov}}
\label{app:cov}
Let $P_B(\cdot)$ be the Poisson process with rate $\nu$, corresponding to the occurences of shots, and let ${\mathcal E}_{t,\delta}(n)$ be the event that $P_B(t+\d)-P_B(t)=n$. By conditioning on the number of shots in the interval $(t,t+\d]$, we find
\begin{align*}
\E \L_1(t) \L_2(t+\d) &= \sum_{n=0}^\infty \E(\L_1(t) \L_2(t+\d) \,|\, {\mathcal E}_{t,\delta}(n)) \Pb({\mathcal E}_{t,\delta}(n))\\
&= \sum_{n=0}^\infty \E(\L_1(t) \L_2(t+\d)\, |\,{\mathcal E}_{t,\delta}(n)) \;  \f{ (\d \nu)^n}{n!} e^{-\nu\d}.
\end{align*}

We proceed by rewriting the conditional expectation as
\begin{align*}
\E(\L_1(t) \L_2(t+\d) \,|\, {\mathcal E}_{t,\delta}(n) )= \f{1}{\delta^n} \Int_{t}^{t+\delta} \dots \Int_{t}^{t+\delta} \E(\L_1(t) \L_2(t+\d) \,|\, \mathcal{F}_{t_1,\ldots,t_n,\delta}(n)) \dif t_1  \dots \dif t_n,
\end{align*}
denoting by $\mathcal{F}_{t_1,\ldots,t_n,\delta}(n)$  the event ${\mathcal E}_{t,\delta}(n)$ and the arrival epochs are $t_1,\ldots,t_n$. Note that we have due to Eqn.\ \eqref{eq:SN}, conditional on $\mathcal{F}_{t_1,\ldots,t_n,\delta}(n)$, the distributional equality
\begin{equation}
\label{eq:disteq}
\L_2(t+\delta) = \L_2(t)e^{-r_2\d} + \sum_{i=1}^n B_{i2} e^{-r_2(t+\delta-t_i)},
\end{equation}
and consequently
\begin{equation}
\label{eq:tbi}
\E\left(\L_1(t)\L_2(t+\delta) \,|\, \mathcal{F}_{t_1,\ldots,t_n,\delta}(n) \right)=\E \L_1(t)\L_2(t) e^{-r_2\d} + \E \L_1(t) \sum_{i=1}^n \E B_{i2} e^{-r_2(t+\delta-t_i)}.
\end{equation}
Note that for all $i=1,\ldots,n$ we have
\begin{align}
\label{eq:int}
 \Int_{t}^{t+\delta} \dots \Int_{t}^{t+\delta}\Int_{t}^{t+\delta} e^{-r_2(t+\delta-t_i)}  \dif t_1 \dif t_2 \dots \dif t_n = \f1{r_2}(1-e^{-r_2\d}) \d^{n-1}.
\end{align}
After unconditioning Eqn.\ \eqref{eq:tbi} with respect to the arrival epochs by integrating over all $t_i$ from $t$ to $t+\d$ and dividing by $\d^n$, we thus obtain
\[
\E(\L_1(t) \L_2(t+\delta) \,| \,{\mathcal E}_{t,\delta}(n) )=\E \L_1(t) \E\L_2(t) e^{-r_2\d} 
+ \E \L_1(t) \f{1}{r_2\d}(1-e^{-r_2\d})n \E B_{12}
\]
and hence, denoting $\L_i:=\lim_{t\to\infty}\L_i (t)$ for $i=1,2$,
\begin{eqnarray*}
\E \L_1(t) \L_2(t+\d) &=& \sum_{n=0}^\infty \left(\E\L_1(t)\L_2(t)e^{-r_2\d} + \E \L_1(t)\f{1}{r_2\d}(1-e^{-r_2\d}) n \E B_{12}\right) \frac{(\d \nu)^n}{n!} e^{-\nu\d}\\
%&=& \E\L_1(t)\L_2(t) e^{-r_2\d} + \E \L_1(t) \f{1}{r_2\d}[1-e^{-r_2\d}] \E B_{12} \sum_{n=1}^\infty \f{\d^n \nu^n}{(n-1)!} e^{-\nu\d}\\
&=&\E\L_1(t)\L_2(t) e^{-r_2\d} + (1-e^{-r_2\d})\E \L_1(t)\E\L_2\\
&=& \E\L_1(t)\E\L_2 + e^{-r_2\d}\big(\E\L_1(t)\L_2(t) - \E\L_1(t)\E\L_2\big),
\end{eqnarray*}
where we made use of $\E\L_i = {\nu\E B_{1i}}/{r_i}$. It follows that
\begin{eqnarray*}
\Cov(\L_1(t), \L_2(t+\d)) &=& \E \L_1(t) \L_2(t+\d) - \E \L_1(t) \E \L_2(t+\d)\\
&=& \E \L_1(t)\L_2(t) e^{-r_2\d} + (1-e^{-r_2\d}) \E \L_1(t) \E \L_2\\
&&-\, \E \L_1(t)\big(\E\L_2(t) e^{-r_2\d} + (1-e^{-r_2\d})\E \L_2 \big),
\end{eqnarray*}
where the equality $\E\L_2(t+\delta)=\E\L_2(t)e^{-r_2\delta} + (1-e^{-r_2\delta})\E\L_2$ is used, which can be directly checked using the expressions for the mean in Eqn.\ \eqref{eq:momentsSN}. This proves the first equality in Eqn.\ \eqref{eq:covl1l2}. The proof of the second equality follows from $\Cov(\L_1(t),\L_2(t))=\E \L_1(t)\L_2(t) - \E \L_1(t)\E\L_2(t)$, in which
\begin{align*}
&\E \L_1(t) \L_2(t) =\E\left[\left(\sum_{i=1}^{N(t)} B_{i1} e^{-r_1(t-U_i)}\right)\left(\sum_{j=1}^{N(t)} B_{j2} e^{-r_2(t-U_j)}\right)\right]\\
&=\sum_{n=0}^\infty e^{-{\nu t}}\f{(\nu t)^n}{n!} \sum_{i=1}^n \sum_{j=1}^n \E(B_{i1} B_{j2} e^{-r_1(t-U_i)}e^{-r_2(t-U_j)})\\
&=\sum_{n=0}^\infty e^{-{\nu t}}\f{(\nu t)^n}{n!} \Big(\f n t \E(B_{11}B_{12}) \int_0^t e^{-(r_1+r_2)(t-u)}\dif u\\
&+ \E B_{11} \E B_{12} \f{n(n-1)}{t^2} \int_0^t e^{-r_1(t-u)}\dif u \int_0^t e^{-r_2(t-v)}\dif v\Big)\\
&=\f{\nu^2\E B_{11}\E B_{12}(1-e^{-r_1 t})(1-e^{-r_2t})}{r_1r_2} + \f{\nu\E B_{11}B_{12}}{r_1+r_2}(1-e^{-(r_1+r_2)t}),
\end{align*}
and $\E\L_i(t)$, for $i=1,2$, is given in Eqn.\ \eqref{eq:momentsSN}.
\end{appendices}

{\small
\bibliographystyle{plain}
\bibliography{biblio}}

\end{document}